\documentclass[12pt]{article}
\usepackage{verbatim}
\usepackage{amstext}
\usepackage{amsthm}
\usepackage{amsmath}
\usepackage{amssymb}
\usepackage{latexsym}
\usepackage{amsfonts}
\usepackage{color}
\usepackage{graphicx}
\usepackage{mathrsfs}

\definecolor{dark_purple}{rgb}{0.4, 0.0, 0.4}
\definecolor{dark_
}{rgb}{0.0, 0.7, 0.0}

\numberwithin{equation}{section}

\title{Complete Norm Preserving Extensions of Holomorphic Functions}
\author{Jim Agler
\and
\L{}ukasz Kosi\'nski
\thanks{Partially supported by the NCN grant SONATA BIS no. 2017/26/E/ST1/00723.
Also funded by the Priority Research Area SciMat under the program 
Excellence Initiative – Research University at the Jagiellonian University in Kraków}
\and
John E. M\raise.5ex\hbox{c}Carthy
\thanks{Partially supported by National Science Foundation Grant  
DMS 2054199}
}
\date{\today}

\newcommand{\mc}{M\raise.45ex\hbox{c}Carthy}

\let\i=\infty

\def\={\ = \ }

\def\d{\mathbb{D}}
\def\h{\mathcal{H}}

\def\be{\begin{equation}}
\def\ee{\end{equation}}

\def\D{\mathbb{D}}

\def\c{\mathbb{C}}
\def\d{\mathbb{D}}
\def\t{\mathbb{T}}

\def\k{\mathcal{K}}
\def\h{\mathcal{H}}

\def\l{\mathcal{L}}

\def\hol{{\rm Hol}}

\def\ess{\mathscr{S}}

\def\O{\Omega}
\def\l{\lambda}
\newcommand\C{{\mathbb C}}
\newcommand\T{{\mathcal T}}

\newcommand\Do{{\mathcal D}_0}

\def\be{\begin{equation}}
\def\ee{\end{equation}}
\def\beq{\begin{eqnarray*}}
	\def\eeq{\end{eqnarray*}}
\def\vs{\vskip 5pt}

\def\oec{{}{\hspace*{\fill} $\lhd$} \vskip 5pt \par}
\def\bl{\begin{lemma}}
	\def\el{\end{lemma}}
\def\bt{\begin{theorem}}
	\def\et{\end{theorem}}
\def\bprop{\begin{prop}}
	\def\eprop{\end{prop}}
\def\bd{\begin{definition}}
	\def\ed{\end{definition}}
\def\br{\begin{remark}}
	\def\er{\end{remark}}
\def\bexer{\begin{exercise}}
	\def\eexer{\end{exercise}}

\begin{document}

\bibliographystyle{plain}
\theoremstyle{definition}
\newtheorem{defin}[equation]{Definition}
\newtheorem{lem}[equation]{Lemma}
\newtheorem{prop}[equation]{Proposition}
\newtheorem{thm}[equation]{Theorem}
\newtheorem{claim}[equation]{Claim}
\newtheorem{ques}[equation]{Question}
\newtheorem{remark}[equation]{Remark}
\newtheorem{fact}[equation]{Fact}
\newtheorem{axiom}[equation]{Technical Axiom}
\newtheorem{newaxiom}[equation]{New Technical Axiom}
\newtheorem{cor}[equation]{Corollary}
\newtheorem{exam}[equation]{Example}
\maketitle

\begin{abstract}
We show that for every connected analytic subvariety $V$ there is a pseudoconvex set $\O$ such that every bounded matrix-valued holomorphic function on $V$ extends isometrically to $\O$. We prove that if $V$ is two analytic disks intersecting at one point, if every bounded scalar valued holomorphic function extends isometrically to $\O$, then so does every matrix-valued function. In the special case that $\O$ is the symmetrized bidisk, we show that this cannot be done by finding a linear
isometric extension from the functions that vanish at one point.
\end{abstract}


\section{Introduction}
By a {\em Cartan pair} we mean a pair $(\O,V)$ where 
$\Omega$ is a connected pseudo-convex set in $\C^n$ and $V$ is an analytic subvariety of $\Omega$.
The name is homage to H. Cartan, who  proved that every holomorphic function on $V$ (i.e. a function that locally agrees with the restriction of a holomorphic function defined on an
open set in $\c^n$) extends to a holomorphic function on all of $\O$ \cite{car51}. We say that a pair
$(\O,V)$ is a {\em norm preserving pair} (np pair for short) if it is a Cartan pair with the additional property that every bounded 
holomorphic function on $V$ extends isometrically to a bounded holomorphic function on $\O$. 

For a fixed domain $\O$, several papers have studied what analytic subvarieties gave rise to np pairs \cite{agmcvn, ghw08, maci19,kmc19, kmc20}. If $\O$ is suitably nice, the conclusion of these papers was that $V$ had to be a holomorphic retract of $\O$ for $(\O,V)$ to be an np pair. 
However, this is not true in general. The simplest example is the np pair $(\Delta,T)$,
where $\Delta$   is the diamond
$\{z\in \c^2:\ |z_1|+|z_2|<1\}$,  and $T =(\d\times \{0\})\cup (\{0\}\times \d)$.

In \cite{akmc21}, the perspective was shifted, to start with $V$ and try to find a pseudoconvex set $G$ so that
$(G,V)$ forms an np pair.
We showed this can always be done:
\begin{thm}
\label{thma1} \cite{akmc21}
	If $(\Omega,V)$ is a Cartan pair, then their exists $G$ such that $(G,V)$ is an np pair.
\end{thm}

The first goal of this note is to extend Theorem \ref{thma1} to the matrix and operator-valued case.

\begin{defin}
	Let $G$ be a domain of holomorphy, and $V$ an analytic subvariety of $G$.  We say $(G,V)$ is a {\em complete
		np pair} if for every separable Hilbert space $\h$ and every bounded holomorphic function $f: V \to B(\h)$ 
	there is a bounded holomorphic extension $F: G \to B(\h)$ such that $\| F \|_G = \| f \|_V$.
\end{defin}

\begin{thm}\label{exist.lem.10}
	If $(\Omega,V)$ is a Cartan pair, and $V$ is connected, then their exists $G$ such that $(G,V)$ is a complete np-pair.
\end{thm}
We prove Theorem \ref{exist.lem.10} in Section \ref{secc}.
Since any Stein manifold embeds properly as a submanifold into $\mathbb C^n$ for some $n$, the theorem carries over to the case when $V$ is a subvariety of a Stein manifold.
Notice that if $V$ is not connected, the characteristic function of any component cannot be isometrically extended to any connected domain containing it, so the connectedness condition is necessary.

We do not know the answer to the following question.
\begin{ques}
\label{q1}
If $(G,V)$ is an np pair  is it always a complete np pair?
\end{ques}

In Section \ref{secd} we study the question for a particular type of $V$, namely one that looks like two crossed discs.
\begin{thm}\label{twocrossed}
	Let $\T$ be the union of  two analytic disks, which intersect at one point $a$.
	\begin{equation}
	\label{eqa2}
	D_1 = \psi_1({\mathbb D}), D_2 = \psi_2({\mathbb D}),  \T = D_1 \cup D_2 , D_1 \cap D_2 = a = \psi_1(0) = \psi_2(0) .
	\end{equation}
	Let $(G,\T)$ be a Cartan pair. 
	Then the following are equivalent:
	
	(i) There is a map $\alpha : {\mathbb T}^2 \to H^\infty_1(G)$
	so that 
	\begin{eqnarray*}
		\alpha(\tau_1, \tau_2) (\psi_1(z) ) & =&  \tau_1 z \\
		\alpha(\tau_1, \tau_2) (\psi_2(z) ) & =&  \tau_2 z.
	\end{eqnarray*}
	
	(ii) $(G,\T)$ is an np pair.
	
	(iii) $(G,\T)$ is a complete np pair.
\end{thm}


We shall let $H^\i(V)$ denote the algebra of bounded holomorphic functions on $V$ equipped with the supremum norm.
\begin{defin}A Cartan pair $(G,V)$ is said to be a \emph{linear np pair} if there is a linear and isometric map $ H^\i(V)\to H^\i(G).$ It is a \emph{ linear np pair vanishing at $a$} if there is a linear and isometric map from the subspace of $H^\i(V)$ that vanishes at $a$ to $H^\i(G)$.
\end{defin}
The linear extension property was first studied by W. Rudin \cite{rud69}.
There is a natural connection between the linear and complete extension properties.
We show in Proposition \ref{propc6} that if $(G,V)$ is a linear np pair vanishing at some point $a$, then $(G,V)$ is a complete np pair.

\vs
{\bf Proposition \ref{propc6}.}
	Let $(\Omega,V)$ be a Cartan pair, $a\in V$, and assume that there is an isometric linear operator $$E:\ess_{a}(V)\to \ess_{a}(\Omega).$$ Then $(\Omega,V)$ is a complete np-pair.
\vs

 In \cite{aly19}  
Agler, Lykova and Young studied the symmetrized bidisc
 \[
 \mathbb G_2\ =\{(z+w, zw):\ z,w\in \d\} .\]
This is  $\c$-convex, though not convex, and
 there are np sets that are not retracts. More precisely they showed that all algebraic sets $V$ in the symmetrized bidisc that have
 the norm preserving  extension property are either retracts or are the union of two analytic discs of the form
\begin{equation}\label{sym}
\{(2\l, \l^2):\l \in \d\}\cup \{(\beta + \bar\beta \l, \l):\l \in \d\},
\end{equation} where $\beta\in \d$.
It follows from Theorem \ref{twocrossed} that for algebraic sets in $ \mathbb G_2$, the np property and the complete np property are the same.
However this cannot be deduced using a linear extension, as we shall show in Theorem \ref{G2linear} that if $V$ is as in \eqref{sym}, there is no linear isometric extension operator of the functions vanishing at a point to all of $\mathbb G_2$.
\vs

{\bf Theorem \ref{G2linear}.}
Let $\T$ be given by \eqref{sym}, and 
let $a \in \T$. There is no 
 linear isometric extension
	operator from $\ess_{a}(\T)$ to $\ess(\mathbb G_2)$.
\vs

%
%
%
%
%
%

\section{Notation}
\label{secb}

If $V$ is any set on which we can define holomorphic functions, we define the {\em Schur class} $\ess(V)$ to be the holomorphic functions from $V$ to $\overline{\D}$. If $\h$ is a Hilbert space, we let $\ess(V, B(\h))$ denote the holomorphic functions from
$V$ to $B(\h)$ that are bounded by $1$ in norm. Finally, if $a \in V$, we let $\ess_a(V)$ (resp. $\ess_a(V, B(\h))$ ) denote
the Schur functions that vanish at $a$.

We define the map  $\pi : \D^2 \to {\mathbb G}$ by
\begin{equation}
\label{eqb1}
\pi(z_1, z_2) \= (s,p) \=  (z_1 + z_2, z_1 z_2) . 
\end{equation}
Define $\Delta$ and $T$ by
\begin{eqnarray}
\label{eqb3}
\Delta &\=& \{z\in \c^2:\ |z_1|+|z_2|<1\}\\
\label{eqb2}
T &\= & \D \times \{ 0 \} \cup \{ 0 \} \times \D .
\end{eqnarray}

\section{Complete np pairs}
\label{secc}

Throughout this section we shall assume that $(\O,V)$ is a Cartan pair and that $V$ is connected.

It was proved by  Bishop \cite{bis62} and Fujimoto \cite{fu65} that if $(\O,V)$ is a Cartan pair, then every $B(\h)$-valued holomorphic function on $V$ extends to a $B(\h)$-valued holomorphic function on $\O$. With this tool in hand one could try to prove Theorem \ref{exist.lem.10} by  repeating the proof from one dimensional case. The main problem that appears here is that Montel's theorem fails for holomorphic functions with values in infinite dimensional vector spaces. There are topologies on $B(\h)$ for which a Montel-type theorem does hold and even such that $\hol(V, B(\h))$ is paracompact, but then the projection $$\hol(\Omega, B(\h))\to \hol(V, B(\h))$$ is not not open, so Michael's selection theorem cannot be used. 
So we shall adopt a new strategy,   which is  to  establish a link between complete and linear norm preserving extensions.

Recall that $f:\Omega \to B(\h)$ is holomorphic if and only if it is weakly holomorphic, i.e. $\Lambda(f)$ is a holomorphic function for any $\Lambda\in B(\h)'$. If $f$ is locally bounded, a weaker condition needs to be verified for a function to be holomorphic: 
\begin{lem}\label{hol}
If $G$ is  open, and $f: G \to B(\h)$ is locally bounded, then $f$ is holomorphic if and only if $z\mapsto \langle f(z)h, k\rangle$ is a holomorphic function.
\end{lem}

\begin{proof}
This can be proved in a similar way to proving that weakly holomorphic functions are holomorphic. See \cite[Thm. 6.1]{kr20} for a formal proof.
\end{proof}

The Montel theorem fails in $\hol(\Omega, B(\h))$. However, it is true if we equip $B(\h)$ with the WOT topology:
\begin{lem}\label{mont}
	Let $(f_n)\subset \ess(\Omega,B(\h))$. Then there is a subsequence $(f_{n_k})$ and $f\in \ess(\Omega, B(\h))$ such that $\langle f_{n_k}(z) h, k\rangle$ converges to $\langle f(z) h,k\rangle$ locally uniformly on $\Omega$ for each $h,k\in \h$.
\end{lem}
\begin{proof}
	For an orthonormal basis $e_i$ we apply the regular Montel theorem
	to $\langle f_n(z) e_i, e_j \rangle$, and then using a Cantor diagonal argument we end up with $f$. 
	
	It is elementary to see that $f$ satisfies desired properties.
\end{proof}
For a final proof we need a few more preparatory results.
\begin{lem}\label{exist.tran}
	 Fix a point $a \in V$.
	Then $(\Omega,V)$ is completely norm preserving if and only if each $f \in \ess_{a}(V,B(\h))$ has an extension to an element $F \in \ess(\Omega, B(\h))$.
\end{lem}
\begin{proof}

If $||f(a)||<1$, there exists an automorphism $m$ of the unit ball of  $B(\h)$ such that $(m \circ f) (a)=0$ \cite{Ha73}. As $h=m \circ f \in \ess(V, B(\h))$, the assumption of the lemma implies that there exists $H \in \ess(\Omega, B(\h))$ such that $H|V=h$. But then if we define $F=m^{-1} \circ H$, $F \in \ess(\Omega, B(\h))$ and $F|V =f$.

If $||f(a)||=1$, we approximate $f$ uniformly with $f_n\in \ess(\Omega, B(\h))$ such that $||f_n(a)||<1$
(eg. $f_n = \frac{n-1}{n} f$). It follows from the previous case that there are $F_n\in \ess(\Omega, B(\h))$ that extend $f_n$. Applying Lemma~\ref{mont} to $F_n$ we find $F\in \ess (\Omega, B(\h))$ that clearly extends $f$.
\end{proof}

\begin{lem} \cite[Lem. 3.3]{akmc21}
\label{lemmont}
	If $(\Omega,V)$ is a Cartan pair and $a\in V$, then $\ess_a(V)$ is a compact subset of
	$\mathcal O (V)$.
	\end{lem}

The following result was proved in \cite[Thm. 3.5]{akmc21}.
\begin{lem}\label{prev}
	If  $a\in V$,  there is a continuous function $S:\hol(V) \to \hol(\Omega)$ such that $S(f)|V =f$ for $f\in \mathcal O(V)$. Moreover, for each $a \in V$ there is an open $G\subset \Omega$ such that $(G,V)$ is a Cartan pair and $S(\ess_{a}(V))\subset \ess_{a}(G)$.
\end{lem}

With this tool in hand we can prove the following linear extension result. Let  $L^2_h(G)$ denote 
the weighted Bergman space obtained from using the Gaussian measure. (If $G$ has finite volume, we could
just use the standard Bergman space).

%
%

\begin{lem}
\label{lemc7}
Fix $a\in V$. Then there is a pseudoconvex domain $D$, $V\subset D\subset \Omega$, and a linear isomorphic extension map $$\ess_{a}(V)\to \ess_{a}(D). $$
\end{lem}

\begin{proof}
	Let $G$ and $S:\hol(V)\to \hol(G)$ be as in Lemma~\ref{prev}. 
	
	 Then the inclusion $\iota : \ess_{a} (G)\subset L_h^2(G)$ is continuous; 
	 composing with $S$ we get a continuous extension operator $$\iota \circ S :\ess_{a}(V) \to L_h^2(G).$$
	 Let $P$ be the orthogonal projection from $L^2_h (G)$ onto $\{g\in L_h^2(\Omega):\ g|_V=0\}^\perp$.
	 Then \[
	 E(f) \= P [\iota \circ S (f)]
	 \]
	 is the 
 element in $L_h^2(G)$ that extends $f$ and has minimal norm. It is straightforward to see that it is linear. By the Cauchy formulas the inclusion $L_h^2(G)\subset \mathcal O(G)$ is continuous. Thus, we can construct a continuous and linear extension operator (which we will also call $E$) $$E:\ess_{a}(V) \to \mathcal O(G).$$
	
Define
	$$D_1 := \left(\bigcap\{z\in G: |E(f)(z)|<1\ \forall f\in \ess_{a}(V)\}\right)^\circ.$$ 
To see $V \subset D_1$, suppose $b \in V \setminus D_1$. Then there exist sequences $b_n \in G$ converging to $b$,
and $f_n$ in $\ess_{a}(V)$, such that $|E(f_n) (b_n)| \geq 1$.
By Lemma \ref{lemmont}, some subsequence of $(f_n)$ converges to a function $f \in \ess_{a}(V)$.
Since $E$ is continuous, $|E(f) (b)| \geq 1$. This would violate the maximum principle.

		Define $D$ to be the connected component of $D_1$ that contains $V$.
	By \cite[Prop. 4.1.7]{jj20}, 
 $D$ is pseudoconvex.
\end{proof}

\begin{prop}
\label{propc6}
	Let $(\Omega,V)$ be a Cartan pair, $a\in V$, and assume that there is an isometric linear operator $$E:\ess_{a}(V)\to \ess_{a}(\Omega).$$ Then $(\Omega,V)$ is a complete np-pair.
\end{prop}

\begin{proof}
	It follows from Lemma~\ref{exist.tran} that it is enough to show that any mapping in $\ess_{a}(V, B(\h))$ has an extension to $\ess(\Omega, B(\h))$. So fix $f\in \ess_{a}(V, B(\h))$ and $z\in \Omega$. Applying the Riesz reprezentation theorem to
	the maps $$\h\ni k\mapsto E(\langle f(\cdot) h, k\rangle)(z),$$ where $h\in \h$, we get for each $z \in \O$ and $h \in \h$,
	a vector 
	$\Psi(z,h)\in \h$ such that 
	\[
	 E(\langle f(\cdot) h, k\rangle)(z) \= \langle \Psi(z,h), k\rangle.
	 \] Note that $h\mapsto \Psi(z,h)$ is linear since $E$ is, 
	 so we can define $F(z):\h\to \h$ by $F(z)h=\Psi(z,h)$. It is straightforward to check that $F(z)\in B(\h))$ and $||F(z)||\leq 1$. Since  $z\mapsto F(z)$ is holomorphic by Lemma \ref{hol}, we are done.
\end{proof}

Combining Proposition \ref{propc6} and Lemma \ref{lemc7}, we have proved Theorem \ref{exist.lem.10}.

\section{Two crossed discs}
\label{secd}

In this section we shall prove Theorem~\ref{twocrossed}. When $\h$ is one dimensional the result was essentially proved in \cite{akmc21}. The key argument used there relied on the Herglotz representation theorem.
 To go to infinite dimensions, we shall use  realization formulas.

\begin{proof}[Proof of theorem~\ref{twocrossed}]
The implications (iii) $\Rightarrow $ (ii) $\Rightarrow$ (i) are trivial.
Let us show (i) $\Rightarrow$ (iii).

Let $\varphi\in \ess(\T,B(\mathcal H))$. By Lemma \ref{exist.tran}, we can assume that $\varphi(a)=0$.
By using the newtwork realization formula (\cite[Thm. 3.16]{amy20}) for the functions
 $$\l\mapsto  \frac{\varphi(\psi_1(\l))}{\lambda} \text{ and } \lambda\mapsto \frac{\varphi(\psi_2(\l))}{\lambda},$$
 we get Hilbert spaces $\mathcal K_1$, $\mathcal K_2$ and unitary operators $U_1=\left(\begin{array}{cc} A_1 & B_1 \\ C_1 & D_1 \end{array}\right):\mathcal H\oplus \mathcal K_1\to \mathcal H\oplus \mathcal K_1$ and $U_2=\left(\begin{array}{cc} A_2 & B_2 \\ C_2 & D_2 \end{array}\right):\mathcal H\oplus \mathcal K_2\to \mathcal H\oplus \mathcal K_2$ such that 
 $$\varphi(\psi_1(\l)) = A_1\lambda + B_1\lambda (I - D_1\lambda)^{-1} C_1\lambda$$ and
$$\varphi(\psi_2(\l)) = A_2\lambda + B_2\lambda (I - D_2\lambda)^{-1} C_2\lambda.$$ Replacing $U_1$ with $\left(\begin{array}{ccc} A_1 & B_1 & 0\\ C_1 & D_1 &0 \\ 0 & 0 & I_{\mathcal K_2} \end{array}\right)$ and $U_2$ with $\left(\begin{array}{ccc} A_2 & 0 & B_2 \\ 0 & I_{\mathcal K_1} & 0\\ C_2 & 0 & D_2 \end{array}\right)$, we have that $$U_1,U_2:\mathcal H \oplus \mathcal K\to \mathcal H \oplus \mathcal K,$$ where $\mathcal K = \mathcal K_1 \oplus \mathcal K_2$.


Let ${\mathcal M} = \mathcal H \oplus \mathcal K$.
Consider the following holomorphic map $\T\to B(\mathcal M)$
$$(\dag)\begin{cases} \psi_1(\lambda)\mapsto \l U_1,\\
\psi_2(\l) \mapsto \l U_2.
\end{cases}
$$

{\bf Claim.} There is a sequence $\Phi_n\in \ess(G, B(\mathcal M))$ that approximates $(\dag)$ in the following sense: $\Phi_n(\psi_1(\lambda)) = \l U_1 \,$ and $\Phi_n(\psi_2(\lambda) )= \l W_n $, where $W_n$ are unitary and converge to $U_2$ in norm.

\bigskip

Observe that the Claim implies the assertion. Indeed, with respect to the decomposition 
$ \mathcal M = \h \oplus \k$, write $$\Phi_n= \left( \begin{array}{cc} \Phi_{1,n} & \Phi_{2,n} \\ \Phi_{3,n} & \Phi_{4,n} \end{array} \right).$$ Then $$f_n(z):= \Phi_{1,n}(z) + \Phi_{2,n}(z) (I - \Phi_{4,n}(z))^{-1} \Phi_{3,n}(z)$$ is an extension from which
 we can take a subsequence converging to the extension we are looking for.

\bigskip

{\em Proof of the claim.}
	Without loss of generality we can assume that $U_1$ is the identity. 
	If  ${\mathcal M}$ is finite dimensional, we use the fact that the eigenvalues of $U_2$ are unimodular, and
	by hypothesis we can extend the function $\tau \lambda$ for any unimodular $\tau$.
	In the infinite dimensional case, choose unitaries $V_n$ that are diagonalizable and converge
	to $U_2$.
	Each $V_n = W_n D_n W_n^*$ where $W_n$ is unitary and $D_n$ is diagonal.
	For each diagonal entry $\tau_k$, let $g_k$ be the Schur function on $G$ that extends the function
	$\psi_1(\lambda) \mapsto \lambda$ and $\psi_2(\lambda) \mapsto \tau_k \lambda$. 
	Then $\Phi_n = W_n D_{g_k} W_n^*$ where $D_{g_k}$ is the diagonal operator with entries $g_k$.
\end{proof}

\section{Linear vs. complete}
\label{sece}

Consider two particular examples:
\begin{enumerate}
	\item The  diamond $\Delta=\{z\in \c^2:\ |z_1|+|z_2|<1\}$ and the two crossed discs $T:=(\D\times \{0\})\cup (\{0\}\times \D).$ 
	\item The symmetrized bidisk $\mathbb G_2$ and the set $\T = \{(2\l, \l^2):\l \in \d\}\cup \{(\beta + \bar\beta \l, \l):\l \in \d\}$ from \eqref{sym}.
\end{enumerate}
It follows from Theorem~\ref{twocrossed} that both $(\Delta, T)$ and  $(\mathbb G_2, \T)$ are complete np-pairs.

Another way to prove this for $(\Delta, T)$ is to observe that the map that sends $f$ in
$\ess_0(T, B(\h))$
to the function $ \{z\mapsto f(z_1,0) + f(0,z_2)\}$ in  $\ess(\Delta, B(\h))$ is linear, and then apply Proposition
\ref{propc6}. We shall show that this argument cannot be used for $(\mathbb G_2, \T)$.

Let us introduce some additional notation before proving this.
Let $\Sigma = \{(2\l, \l^2):\ \l \in \d\}$ and $\Do = \{0\}\times \d$.
Let 
 $[f(\l),g(\l)]$ denote the function on $\Sigma \cup \Do$ that is equal to $f(\lambda)$ on $(2\lambda, \lambda^2)$ and $g(\lambda)$ on $(0,\lambda)$.
For $b\in \mathbb D$ let $m_b$ be a M\"obius map $m_b (\lambda)= \frac{b-\lambda}{1-\bar b \lambda}.$

\begin{thm}\label{G2linear}
Let $\T$ be given by \eqref{sym}, and 
let $a \in \T$. There is no 
 linear isometric extension
	operator from $\ess_{a}(\T)$ to $\ess(\mathbb G_2)$.
\end{thm}
\begin{proof}
	
	Since all sets of the form \eqref{sym} are holomorphically equivalent, it suffices to prove the assertion for $\T=\Sigma\cup \Do$.
	
	For unimodular $\alpha$ and $\beta$ consider the function $f_{\alpha, \beta}: \T \to \mathbb D$ given by the formula:
	$$f_{\alpha, \beta}(s,p) =  \begin{cases} \alpha s/ 2,\quad \text{on $\Sigma$},\\
	\beta p,\quad \text{on $\Do$}.
	\end{cases}$$
So $ f_{\alpha, \beta} =[\alpha \lambda, \beta \lambda] $.
	Let   $\omega = \beta \alpha^{-1}$. It was shown in \cite{aly19} that $\alpha \Phi_\omega$ extends $f_{\alpha, \beta}$ where $$\Phi_{\omega}(s,p):= \frac{s/2 +  \omega p}{1+ \omega s/2}.$$
	
	{\bf Claim.} We shall show that $\alpha \Phi_{\beta \alpha^{-1}}$ is the unique np extension of $f_{\alpha,\beta}$ to $\mathbb G_2 \to \D$:
	
	To prove the claim let $F$ be some extension of  $f_{\alpha, \beta}$ that has norm 1.
 Let $\omega\in \mathbb T$. Then \[
	F(s,p) = \alpha s/2 + \beta (p - (s/2)^2) + O( s (s^2 - 4p)), \]
since $F$ minus the first two terms vanishes on $\T$.
With $\pi$ as in \eqref{eqb1}, we get
	 \[
	F(\pi(\lambda, \omega \lambda)) = \alpha \frac{1+\omega}{2} \lambda - \beta \left(\frac{1-\omega}{2}\right)^2 \lambda^2 + O(\lambda^3).
	\]
Then the Schwarz lemma implies that 	the map
\[
\l \mapsto \frac{1}{\l}
F(\pi(\lambda, \omega \lambda)) \]
 is a M\"obius map from $\D$ to $\D$. Thus if $G$ is another extension, $F\circ \pi$ and $G\circ\pi$ coincide on $\{|\lambda|=|\mu|:\ (\lambda, \mu)\in \D^2\}$, and the claim follows. 
 \oec

	Suppose that there is $a \in \T$ and a linear isometric operator $L : \ess_{a}(\T) \to \ess(\mathbb G_2)$. Let us consider two cases.
	
	i) $a=(0,\l_0)\in\Do$. Note that $\lambda \mapsto [m_{\beta a}(\alpha \lambda), m_{\beta a} (\beta \l)]$ belongs to the Schur class $\ess_{a}(\T)$. The crucial fact following from the Claim is that the equality 
	\be \label{eqe1}
	L[m_{\beta\lambda_0}(\alpha\lambda), m_{\beta\lambda_0} (\beta \lambda)] = m_{\beta \lambda_0}(\alpha \Phi_{\beta/\alpha})
	\ee
 holds for any $\alpha,\beta\in \t$. 
 Writing out \eqref{eqe1}, and using $\omega = \bar \alpha \beta$, we get
 \[
 L \left[  \frac{\beta \lambda_0 - \alpha \lambda}{1 -  \bar \beta \bar \lambda_0 \alpha   \lambda },  \frac{\beta\lambda_0 - \beta \lambda}{1- \lambda \bar \lambda_0}\right]
 \= \frac{\beta \lambda_0 - \alpha \Phi_{\omega}(s,p)}{1 - \bar \beta \bar \lambda_0 \alpha\Phi_{\omega}(s,p)} .
 \]
 Dividing by $\beta$ we get
\be
\label{eqe2}
	L\left[\frac{ \lambda_0 -  \lambda\bar \omega}{ 1-  \lambda \bar \lambda_0\bar \omega },  \frac{\lambda_0 - \lambda}{1- \lambda \bar \lambda_0}\right] =  \frac{ \lambda_0 - \bar \omega  \Phi_{ \omega}(s,p)}{1 - \bar \lambda_0 \bar \omega \Phi_{\omega}(s,p)}.
	\ee
Write
\[
\Phi_\omega(s,p) \= \frac{s/2 + \omega p}{1 + \omega s/2} \= 
\frac{\bar \omega s/2 + p}{\bar \omega + s/2},
\]
and expand both sides of \eqref{eqe2} in powers of $ \bar \omega$.
 Expanding the left hand side we get
  \beq
  \left[\frac{ \lambda_0 -  \lambda\bar \omega}{ 1-  \lambda \bar \lambda_0\bar \omega },  \frac{\lambda_0 - \lambda}{1- \lambda \bar \lambda_0}\right] &\= & \sum_{n\geq 0} \bar \omega^n f_n =\\ 
  &=& \left[\lambda_0, \frac{ \lambda_0 - \l}{ 1- \l \bar \lambda_0}\right] + \bar \omega \left[(| \lambda_0|^2 -1) \l , 0 \right] + \sum_{n\geq 2} \bar \omega^n f_n,
  \eeq
  where $f_n\in H^{\infty}(\T)$, $f_n(a)=0$, and the series converges uniformly, so $L$ can be applied term by term.
The right hand side gives
\[
\frac{ \lambda_0 - \bar \omega  \Phi_{ \omega}(s,p)}{1 - \bar \lambda_0 \bar \omega \Phi_{\omega}(s,p)}
\=
\l_0 - \bar \omega \frac{2p}{s} ( 1 - |\l_0|^2) + O(\bar \omega^2) .
\]
Comparing the constant terms, we would have
\[
L \left[\lambda_0, \frac{ \lambda_0 - \l}{ 1- \l \bar \lambda_0}\right]  = \l_0, 
\]
a contradiction.

	ii) We are left with the case $a=(2\lambda_0, \lambda_0^2)\in \Sigma$, $\lambda_0\neq 0$. We shall proceed as before starting with a function $[m_{\alpha\l_0}(\alpha \l), m_{\alpha \l_0} (\beta \l)]$ that clearly lies in $\ess_{a}(\T)$.
As before, we get that
	\begin{equation}\label{eq:L}
	L\left[\frac{\lambda_0-\lambda}{ 1- \bar \lambda_0 \lambda}, \frac{\lambda_0 - \omega \lambda}{1 - \bar\lambda_0 \omega \lambda}\right] = \frac{\lambda_0 - \Phi_\omega(s,p)}{1 - \bar\lambda_0 \Phi_\omega (s,p)},\quad \omega\in \mathbb T.\end{equation}
	Expanding in powers of $\omega$ and looking at the coefficient of $\omega$, we get
 \be
 \label{eqe4}
 L[0, \lambda] = \frac{p - (s/2)^2}{(1- \bar \lambda_0 s/2)^2}.
 \ee
As $[0,\l]$ lies in $\ess_{a}(\T)$, we must have that the function $(s,p)\mapsto \frac{p - (s/2)^2}{(1- \bar \lambda_0 s/2)^2}$ sends the symmetrized bidisc to the unit disc. In particular, putting $(s,p) = (\l + \mu, \l \mu)$ for  $\lambda, \mu$ in the unit disc we
would get 
 the inequality 
 \be
 \label{eqe5}
 |(\lambda -\mu)/2| \ \leq \  |1 - \bar \lambda_0 (\lambda +\mu)/2|
 \ee
 holds for $(\l,\mu)\in \d^2$. 
This however is not possible whenever $\lambda_0\neq 0$.
Indeed, let $t = |\l_0|$. Then \eqref{eqe5} is equivalent to the claim that
\[
| \l - \mu |^2  \ \leq \ | 2 - t (\l+\mu)|^2 \qquad \forall (\l,\mu) \in \overline{\D^2} ,
\]
since by continuity the inequality would extend to the boundary. Assume both $\l$ and $\mu$ are unimodular, then this becomes
\[
-2 (1+t^2) \Re (\bar \l \mu) + 4t \Re (\l + \mu) \ \leq \ 2 + 2t^2.
\]
Let $\l = e^{i\theta}$ and $\mu = e^{-i\theta}$.
We get the inequality
\be
\label{eq8}
-2 (1+t^2) \cos (2 \theta) + 4t \cos (\theta) \ \leq \ 2 + 2 t^2.
\ee
By calculus, the maximum of the left hand side comes when we
choose $\theta$ so that 
\[
\cos(\theta) = \frac{t}{1+t^2}, \quad \sin(\theta) = \sqrt{ 1 -  \frac{t^2}{(1+t^2)^2}} .
\]
Then \eqref{eq8} becomes
\[
2 \frac{ 1 + 4t^2 + t^4}{1+t^2} \ \leq \ 2 (1+t^2) .
\]
This clearly fails unless $t=0$.
\end{proof}

\bibliography{references_uniform_partial}
\end{document}